\documentclass[a4paper,11pt,reqno]{amsart}
\usepackage{amssymb,amsmath,hyperref}
\title{Two $2/5$-level mock theta conjecture-like identities}

\author{Stepan Konenkov}
\address{Department of Mathematics and Computer Science, Saint Petersburg State University, Saint Petersburg,  Russia, 199178}
\email{konenkov.stepan@yandex.ru}

\author{Eric T. Mortenson}
\address{Department of Mathematics and Computer Science, Saint Petersburg State University, Saint Petersburg,  Russia, 199178}
\email{etmortenson@gmail.com}

\renewcommand\theta{\vartheta}

\newtheorem{theorem}{Theorem}
\newtheorem{lemma}[theorem]{Lemma}
\newtheorem{corollary}[theorem]{Corollary}
\newtheorem{proposition}[theorem]{Proposition}
\theoremstyle{definition}

\newtheorem{remark}[theorem]{Remark}

\numberwithin{theorem}{section} 
\numberwithin{equation}{section}

\newcommand{\Z}{\mathbb{Z}}

\newcommand{\im}{\textnormal{Im}}

\makeatletter
\@namedef{subjclassname@2020}{\textup{2020} Mathematics Subject Classification}
\makeatother

\begin{document}

\date{1 June 2025}



\begin{abstract}
Determining the explicit forms and modularity for string functions and branching coefficients for Kac--Moody algebras after Kac, Peterson, and Wakimoto is an important problem.  In a pair of papers, Borozenets and Mortenson determined the explicit forms for fractional-level string functions for the Kac--Moody algebra $A_{1}^{(1)}$.  For positive fractional-level string functions they obtained mock theta conjecture-like identities, and for negative fractional-level string functions, they obtained mixed false theta function expressions.   Here we find two new families of mock theta conjecture-like identities but for the $2/5$-level string functions.   Each of these two families of identities is composed of the four tenth-order mock theta functions from Ramanujan's Lost Notebook as well as a simple quotient of theta functions. 
\end{abstract}

\maketitle

 \tableofcontents

\section{Introduction}

Determining the explicit forms and modularity for string functions and branching coefficients for Kac--Moody algebras after Kac, Peterson, and Wakimoto \cite{KP84, KW88pnas, KW88advmath, KW89, KW90}  is an important problem.  In a pair of papers, Borozenets and Mortenson \cite{BoMo2024, BoMo2025P2} determined the explicit forms for fractional-level string functions for the Kac--Moody algebra $A_{1}^{(1)}$ \cite{ACT, FZ85, SW}.   In the first paper \cite{BoMo2024}, they used Hecke-type double-sum expansions and mock modularity.  In the second paper \cite{BoMo2025P2}, they developed a notion of quasi-periodicity and then computed polar-finite decompositions of the admissible characters after Zagier--Zwegers \cite{DMZ, Zw02}.   For positive fractional-level string functions they obtained mock theta conjecture-like identities, and for negative fractional-level string functions, they obtained mixed false theta function expressions.  

\smallskip
Here we continue to develop the methods of the second paper and utilize polar-finite decomposition of admissible characters.  We recall that Dabholkar, Murthy, and Zagier \cite{DMZ}  extended results of Zwegers \cite{Zw02} and introduced a canonical decomposition of a meromorphic Jacobi form into a ``polar'' part and a ``finite'' part.  Here the former is completely determined by the poles of the meromorphic Jacobi form; whereas the latter is a finite linear combination of theta functions with mock modular forms as coefficients.  Here we use the polar-finite decomposition to find two new families of mock theta conjecture-like identities but for the $2/5$-level string functions.   Each of these two families of identities is composed of the four tenth-order mock theta functions from Ramanujan's Lost Notebook as well as a simple quotient of theta functions. 

\smallskip
We will focus on the case of the Kac--Moody algebra $A_1^{(1)}$. We let $p \geq 1$, $p^{\prime} \geq 2$ be coprime integers, and we define the admissible level to be
\begin{equation} \label{equation:admlevel}
N:=\frac{p^{\prime}}{p}-2.
\end{equation}

Let $q := e^{2\pi i \tau}$ with $\im(\tau) >0$ and $z \in \mathbb{C}\backslash \{0\}$. 
Using the Weyl--Kac formula, one can express the admissible character as
\begin{equation}\label{equation:WK-formula}
\chi_{\ell}^N(z;q)=\frac{\sum_{\sigma=\pm 1}\sigma \Theta_{\sigma (\ell+1),p^{\prime}}(z;q^{p})}
{\sum_{\sigma=\pm 1}\sigma\Theta_{\sigma,2}(z;q)},
\end{equation}
where theta function is defined
\begin{equation}\label{equation:SW-thetaDef}
\Theta_{n,m}(z;q):=\sum_{j\in\mathbb{Z}+n/2m}q^{mj^2}z^{-mj}.
\end{equation}
However, we will later use a different definition for the theta function.  

\smallskip
Using \eqref{equation:WK-formula}, Kac and Wakimoto \cite{KW88advmath} showed that the characters form a vector-valued Jacobi form.  Furthermore, we can define the string functions for level $N$, quantum number $m$, and spin $\ell$ with $m+\ell \equiv 0 \pmod 2$, through
\begin{equation} \label{equation:fourcoefexp}
\chi_{\ell}^N (z,q)=\sum_{m\in 2\mathbb{Z}+\ell}
C_{m,\ell}^{N}(q) q^{\frac{m^2}{4N}}z^{-\frac{1}{2}m}.
\end{equation}

We have the following symmetries for string functions \cite[(3.4), (3.5)]{SW}, \cite[(2.40)]{ACT}
\begin{equation*}
C_{m,\ell}^{N}(q) = C_{-m,\ell}^{N}(q), \ C_{m,\ell}^{N}(q) = C_{N-m,N-\ell}^{N}(q).
\end{equation*}
For the integral-level $N$ we have the periodicity property \cite[(3.5)]{SW}
\begin{equation} \label{eq:inglevelperiod}
C_{m,\ell}^{N}(q) = C_{m+2N,\ell}^{N}(q),
\end{equation}
and hence from \eqref{equation:fourcoefexp} the theta-expansion
\begin{equation}\label{eq:intlevelthetadecomp}
\chi_{\ell}^N(z,q)=\sum_{\substack{0\le m <2N\\m \in 2\Z + \ell}}C_{m,l}^{N}(q)\Theta_{m,N}(z,q).
\end{equation}
Subsequently, we will recall a quasi-periodic notion for fractional-level string functions and the polar-finite decomposition for the admissible characters.

To state some examples of integral-level string functions, let us first recall the $q$-Pochhammer notation
\begin{equation*}
(x)_n=(x;q)_n:=\prod_{i=0}^{n-1}(1-q^ix), \ \ (x)_{\infty}=(x;q)_{\infty}:=\prod_{i\ge 0}(1-q^ix),
\end{equation*}
and the theta function
\begin{equation}\label{equation:JTPid}
j(x;q):=(x)_{\infty}(q/x)_{\infty}(q)_{\infty}=\sum_{n=-\infty}^{\infty}(-1)^nq^{\binom{n}{2}}x^n,
\end{equation}
where the last equality is the Jacobi triple product identity.   We will frequently use the notation,
\begin{equation*} 
J_{a,b}:=j(q^a;q^b),
 \ \overline{J}_{a,b}:=j(-q^a;q^b), \ {\text{and }}J_a:=J_{a,3a}=\prod_{i\ge 1}(1-q^{ai}),
\end{equation*}
where $a,b$ are positive integers.  

One notes that the two definitions for theta functions are equivalent via the identity
\begin{equation*}
\Theta_{n,m}(z;q)=z^{-\frac{n}{2}}q^{\frac{n^2}{4m}}j\left ( -q^{n+m}z^{-m};q^{2m}\right).
\end{equation*}
Let us denote the Dedekind eta-function as
\begin{equation*} 
\eta(q) :=q^{1/24}\prod_{n\ge 1}(1-q^{n}).
\end{equation*}
Among more general results, Kac and Peterson \cite{KP84} showed
{\allowdisplaybreaks \begin{gather*}
c^{01}_{01} = \eta(q)^{-1},\\
c^{11}_{11} = \eta(q)^{-2}\eta(q^2),\\
c^{21}_{21} = \eta(q)^{-2} q^{3/40} J_{6,15},\\
c^{40}_{22} = \eta(q)^{-2} \eta(q^6) \eta(q^{12})^{2},\\
c^{40}_{40} - c^{40}_{04} = \eta(q^2)^{-2}.
\end{gather*}}%
Kac and Peterson employed modularity techniques to prove such string function identities \cite[p. 220]{KP84}.  In \cite[Example 1.3]{HM}, we find a closed expression in terms of theta functions for the general integral-level string function.  Additional calculations on integral-level string functions can be found in \cite{Mo24B, MPS}.

\smallskip
In this paper, we build on the methods found in the second paper of Borozenets and Mortenson \cite{BoMo2025P2}, where they developed a notion of quasi-periodicity for fractional-level string functions:

\begin{theorem}\label{theorem:generalQuasiPeriodicity}\cite[Theorem $2.1$]{BoMo2025P2} For $(p,p^{\prime})=(p,2p+j)$, we have the quasi-periodic relation for even spin 
\begin{align*}
& (q)_{\infty}^{3}C_{2jt+2s,2r}^{(p,2p+j)}(q)
 -(q)_{\infty}^{3}C_{2s,2r}^{(p,2p+j)}(q)\\
& \ \ \  = (-1)^{p}q^{-\frac{1}{8}+\frac{p(2r+1)^2}{4(2p+j)}}q^{\binom{p}{2}-p(r-s)-\frac{p}{j}s^2} \\
&\qquad  \times \sum_{i=1}^{t}q^{-2pj\binom{i}{2}-2psi}
\sum_{m=1}^{p-1}(-1)^{m}q^{\binom{m+1}{2}+m(r-p)} \times (q^{m(ji+s-j)}-q^{-m(ji+s)}) \\
&\qquad \times \Big (  j(-q^{m(2p+j)+p(2r+1)};q^{2p(2p+j)} )
  -  q^{m(2p+j)-m(2r+1)}j(-q^{-m(2p+j)+p(2r+1)};q^{2p(2p+j)})\Big ).
\end{align*}
\end{theorem}
\noindent One should compare Theorem  \ref{theorem:generalQuasiPeriodicity} with  the periodic property for integral-level string functions found in (\ref{eq:inglevelperiod}).  Using the notion of quasi-periodicity, they established a polar-finite decomposition for admissible characters.

To state the polar-finite decomposition we use the notation of Appell functions, which are the building blocks of Ramanujan's classical mock theta functions \cite[Section 5]{HM}.   Following Hickerson and Mortenson \cite{HM}, we define the Appell function as
\begin{equation}
m(x,z;q):=\frac{1}{j(z;q)}\sum_{r\in\Z}\frac{(-1)^rq^{\binom{r}{2}}z^r}{1-q^{r-1}xz}.\label{equation:m-def}
\end{equation}

It is easier to state polar-finite decompositions in terms of even-spin, and there are cross-spin identities which can be used to write odd-spin decompositions in terms of even-spin decompositions.

\begin{theorem}\label{theorem:generalPolarFinite}\cite[Theorem $2.3$]{BoMo2025P2} For $(p,p^{\prime})=(p,2p+j)$, we have the polar-finite decomposition for even spin 
\begin{align*}
&\chi_{2r}^{(p,2p+j)} (z;q)\\
&\ \  =\sum_{s=0}^{j-1}z^{-s}q^{\frac{p}{j}s^2}C_{2s,2r}^{(p,2p+j)}(q)j(-z^{j}q^{p(j-2s)};q^{2pj})\\
&\qquad +\frac{1}{(q)_{\infty}^3}\sum_{s=0}^{j-1}
(-1)^{p}q^{-\frac{1}{8}+\frac{p(2r+1)^2}{4(2p+j)}}q^{\binom{p}{2}-p(r-s)}z^{-s}
j(-q^{p(j-2s)}z^{j};q^{2jp})
   \sum_{m=1}^{p-1}(-1)^{m}q^{\binom{m+1}{2}+m(r-p)} \\
&\qquad \qquad    \times \Big (  j(-q^{m(2p+j)+p(2r+1)};q^{2p(2p+j)} )
  -  q^{m(2p+j)-m(2r+1)}j(-q^{-m(2p+j)+p(2r+1)};q^{2p(2p+j)})\Big )\\
&\qquad \qquad   \times
\Big (q^{ms-2ps}m(-q^{jm-2ps},-q^{p(j+2s)}z^{-j};q^{2jp}) + q^{-ms}m(-q^{jm+2ps},-q^{p(j-2s)}z^{j};q^{2jp})\Big ).
\end{align*}
\end{theorem}
\begin{remark}  We see that Theorem (\ref{theorem:generalPolarFinite}) is already in the most general form; indeed, we can easily replace $j$ with $(n-2)p+j$.
\end{remark}

For $j=1$ we have the following immediate corollary.
\begin{corollary}\label{corollary:polarFinite1p} \cite[Corollary $2.5$]{BoMo2025P2} For $(p,p^{\prime})=(p,2p+1)$, $(m,\ell)=(2k,2r)$, we have
\begin{align*}
\chi_{2r}^{(p,2p+1)}&(z;q)\\
&=C_{0,2r}^{(p,2p+1)}(q)
j(-q^{p}z;q^{2p})
 \\
&\qquad +(-1)^{p}q^{-\frac{1}{8}+\frac{p(2r+1)^2}{4(2p+1)}+\binom{p}{2}-rp}\frac{j(-q^{p}z;q^{2p})}{(q)_{\infty}^3}
\sum_{m=1}^{p-1} 
(-1)^{m}q^{\binom{m+1}{2}+m(r-p)}
 \\
&\qquad \qquad \times \left ( j(-q^{m(2p+1)+p(2r+1)};q^{2p(2p+1)}) 
-q^{2m(p-r)}j(-q^{-m(2p+1)+p(2r+1)};q^{2p(2p+1)})\right )
 \\
&\qquad \qquad \times 
 \left ( m(-q^{m},-q^{p}z;q^{2p})
+m(-q^{m},-q^{p}z^{-1};q^{2p})\right ).
\end{align*}
\end{corollary}


 The string functions have a lead fractional exponent that makes them modular.  To remove the fractional exponent and make the $q$-series easier to deal with, we define
\begin{equation*}
s_{\ell}^{N}:=-\frac{1}{8}+\frac{(\ell+1)^2}{4(N+2)}, \ \ 
s_{m, \ell}^{N} := s_{\ell}^{N} - \frac{m^2}{4N}.   
\end{equation*}
We incorporate the above into the additional notation
\begin{equation}
\mathcal{C}_{m,\ell}^{N}(q) := q^{-s_{m,\ell}^{N}}C_{m,\ell}^{N}(q) \in \Z[[q]].
\end{equation} 
In particular, we note that for $N:=\frac{2p+j}{p}-2$, that
\begin{equation}
q^{-\frac{1}{8}+\frac{(\ell+1)^2}{4(N+2)}-\frac{m^2}{4N}}\mathcal{C}_{m,\ell}^{(p,2p+j)}(q)
=C_{m,\ell}^{(p,2p+j)}(q).\label{equation:mathCalCtoStringC}
\end{equation}

\begin{theorem}\cite[Theorem $2.2$]{BoMo2025P2}  For $(p,p^{\prime})=(p,2p+1)$, we have the cross-spin identity for even spin
\begin{align}
(q)^3_{\infty} \mathcal{C}^{(p,2p+1)}_{0,2k}(q) = (q)^3_{\infty} &\cdot (-1)^{p+1} q^{-p(k+1)} q^{\binom{p+1}{2}}\mathcal{C}^{(p,2p+1)}_{1,2p-1-2k}(q)\\
&- \sum_{m=1}^{p} (-1)^m q^{-m(k+1)} q^{\binom{m+1}{2}}j(-q^{-m(2p+1)+p(2p+2k+2)};q^{2p(2p+1)})
\notag \\
&+\sum_{m=1}^{p} (-1)^m q^{mk} q^{\binom{m+1}{2}}j(-q^{-m(2p+1)+p(2p-2k)};q^{2p(2p+1)}).\notag
\end{align}
\end{theorem}

\section{Ramanujan's mock theta functions and new mock theta conjecture-like identities}

In Ramanujan's last letter to Hardy, he gave a list of seventeen so-called mock theta functions.   Each function was defined by Ramanujan as a $q$-series convergent for $|q|<1$.  Although the mock theta functions are not theta functions, they do have asymptotic properties similar to those of ordinary theta functions.  In the letter, one finds third-order, fifth-order, and  seventh-order mock theta functions; however, the term ``order'' is not well-defined.

We recall two of the third-order mock theta functions found in Ramanujan's last letter to Hardy:
\begin{equation*}
f_3(q):=\sum_{n\ge 0}\frac{q^{n^2}}{(-q)_n^2}, \ \ 
\omega_3(q)
:=\sum_{n\ge 0}\frac{q^{2n(n+1)}}{(q;q^2)_{n+1}^2}.
\end{equation*}
We also recall some examples from \cite{BoMo2025P2} which will help to motivate the general shape of our new results.   

\smallskip
For the $1/3$-level string functions Borozenets and Mortenson obtained the family of identities:
\begin{theorem}\label{theorem:newMockThetaIdentitiespP37m0ell2r}\cite[Theorem $2.6$]{BoMo2025P2}
For $(p,p^{\prime})=(3,7)$, $(m,\ell)=(0,2r)$, $r\in \{0,1,2\}$ we have
\begin{align*}
(q)_{\infty}^3&\mathcal{C}_{0,2r}^{1/3}(q)
=(-q)^{-r}\frac{(q)_{\infty}^3}{J_{2}}
\frac{j(-q^{1+2r};q^{14})j(q^{16+4r};q^{28})}
{j(-1;q)J_{28}} \\
&\qquad   -q^{2-2r}\frac{j(q^{6-2r};q^{14})j(q^{26-4r};q^{28})}{J_{28}} \omega_3(-q)
+  \frac{q^{-r}}{2} \frac{j(q^{1+2r};q^{14})j(q^{16+4r};q^{28})}{J_{28}} 
f_{3}(q^2).
\end{align*}
\end{theorem}

For the $2/3$-level string functions Borozenets and Mortenson had two families of identities.  We recall the floor function $\lfloor \cdot \rfloor$.  They showed
\begin{theorem}\label{theorem:newMockThetaIdentitiespP38m0ell2r}\cite[Theorem $2.7$]{BoMo2025P2} For $(p,p^{\prime})=(3,8)$, $(m,\ell)=(0,2r)$, $r\in\{0,1,2,3\}$, we have
\begin{align*}
(q)_{\infty}^3&\mathcal{C}_{0,2r}^{2/3}(q) 
=(-1)^{\lfloor (r+1)/2\rfloor}\cdot
\frac{q^{-r}}{2}\frac{J_{1}^2J_{2}}{J_{4}^2J_{8}}
j(-q^{7-2r};q^{16})j(q^{1+2r};q^{8})\\
&\qquad 
- q^{3-2r}    
\frac{j(q^{7-2r};q^{16})j(q^{30-4r};q^{32})}{J_{32}}
\omega_{3}(-q^{2}) 
 + q^{-r} 
\frac{j(q^{1+2r};q^{16})j(q^{18+4r};q^{32})}{J_{32}} 
 \frac{1}{2}f_{3}(q^4).
 \end{align*}
\end{theorem}

\begin{theorem}\label{theorem:newMockThetaIdentitiespP38m2ell2r}\cite[Theorem $2.8$]{BoMo2025P2}  For $(p,p^{\prime})=(3,8)$, $(m,\ell)=(2,2r)$, $r\in\{0,1,2,3\}$, we have
\begin{align*}
(q)_{\infty}^3\mathcal{C}_{2,2r}^{2/3}(q) 
&=(-1)^{\lfloor (r+1)/2\rfloor}\cdot 
\frac{q^{3-2r}}{2}\frac{J_{1}^2J_{2}}{J_{4}^2J_{32}}j(q^{2+4r};q^{32})j(q^{7-2r};q^{16})\\
&\qquad 
- q^{3-2r}   
\frac{j(q^{7-2r};q^{16})j(q^{30-4r};q^{32})}{J_{32}}
\left (1- \frac{1}{2}f_{3}(q^4) \right )  \\
&\qquad + q^{1-r} 
\frac{j(q^{1+2r};q^{16})j(q^{18+4r};q^{32})}{J_{32}} 
 \left (1-q^2\omega_{3}(-q^{2}) \right ).
\end{align*}
\end{theorem}

\smallskip
While on a visit to Cambridge University in the spring of $1976$, George Andrews discovered a long-forgotten collection of  Ramanujan's notes, which contained hundreds of identities without proofs.  The notes have subsequently acquired the name Ramanujan's Lost Notebook  and among the unproven identities were ten formulas relating the fifth-order mock theta functions, now called the (ex-)mock theta conjectures, which are ten identities where each identity expresses a different fifth-order mock theta function in terms of a building block and a single quotient of theta functions.  This particular building block is the so-called universal mock theta function $g(x;q)$, which is defined
\begin{equation*}
g(x;q):=x^{-1}\Big ( -1 +\sum_{n=0}^{\infty}\frac{q^{n^2}}{(x;q)_{n+1}(q/x;q)_{n}} \Big ).
\end{equation*}
Using our notation, two of the ten (ex-)mock theta conjectures read \cite{H1}:
\begin{align*}
f_0(q)&:=\sum_{n= 0}^{\infty}\frac{q^{n^2}}{(-q;q)_n}=\frac{j(q^5;q^{10})j(q^2;q^5)}{j(q;q^3)}-2q^2g(q^2;q^{10}),\\
f_1(q)&:=\sum_{n= 0}^{\infty}\frac{q^{n(n+1)}}{(-q;q)_n}=\frac{j(q^5;q^{10})j(q;q^5)}{j(q;q^3)}-2q^3g(q^4;q^{10}).
\end{align*}

Many identities found in the lost notebook, such as those for the tenth-order mock theta functions, involve a spectacular linear combination of mock theta functions which are evaluated in terms of a single quotient of theta functions.   We recall the two of the four tenth-order mock theta functions \cite{C1, C2, C3, RLN}
\begin{align*}
{\phi}_{10}(q)&:=\sum_{n\ge 0}\frac{q^{\binom{n+1}{2}}}{(q;q^2)_{n+1}}, \ \ {\psi}_{10}(q):=\sum_{n\ge 0}\frac{q^{\binom{n+2}{2}}}{(q;q^2)_{n+1}}.
\end{align*}
The four functions satisfy six identities, each of which involves a single quotient of theta functions.  Letting $\omega$ be a primitive third-root of unity.  One of the six identities reads \cite{C1}
\begin{gather*}
q^{2}\phi_{10}(q^9)-\frac{\psi_{10}(\omega q)-\psi_{10}(\omega^2 q)}{\omega - \omega^2}
=-q\frac{j(q;q^2)}{j(q^3;q^6)}\cdot \frac{j(q^3;q^{15})j(q^{6};q^{18})}{j(q^3;q^9)}.
\end{gather*}
with the other five being similar \cite{C1, C2, C3, RLN}.   What led Ramanujan to such identities is an ongoing mystery. 

\smallskip
Mortenson \cite{Mo24BLMS} discovered three new identities for the sixth-order mock theta functions that are in the spirit of the tenth-order identities and also found and proved nineteen tenth-order like identities for second, sixth, and eighth-order mock theta functions.   Mortenson and Urazov \cite{MoUr24} obtained several families of tenth-order like identities but for Appell functions. 

\smallskip
To further motivate our new results, we recall the $1/5$-level string functions identities \cite{BoMo2025P2}:
\begin{theorem} \label{theorem:newMockThetaIdentitiespP511m0ell2r} \cite[Theorem 2.9]{BoMo2025P2} For $(p,p^{\prime})=(5,11)$, $(m,\ell)=(0,2r)$, $r\in\{0,1,2,3,4\}$, we have
\begin{align*}
(q)_{\infty}^3\mathcal{C}_{0,2r}^{1/5}(q)
&=-q^{r^2-3r+1}J_{1,2}j(q^{4+8r};q^{22})
 \\
&\qquad -
q^{6-4r} \times \left ( j(-q^{16+10r};q^{110}) 
-q^{4+8r}j(-q^{6-10r};q^{110})\right )
 \times q\phi_{10}(-q) \\
& \qquad +q^{3-3r}\times \left ( j(-q^{27+10r};q^{110}) 
-q^{3+6r}j(-q^{17-10r};q^{110})\right )
\times \chi_{10}(q^2) \\
&\qquad 
-q^{1-2r}  \times \left ( j(-q^{38+10r};q^{110}) 
-q^{2+4r}j(-q^{28-10r};q^{110})\right )
\times \left ( -\psi_{10}(-q)\right ) \\
&\qquad  
+q^{-r} \times \left ( j(-q^{49+10r};q^{110}) 
-q^{1+2r}j(-q^{39-10r};q^{110})\right )\times X_{10}(q^2).
\end{align*}
\end{theorem}

\smallskip
Now we can state our new results.  For the $2/5$-level string functions we have two families of results, which are similar in form to \cite[Theorem 2.9]{BoMo2025P2}.
\begin{theorem}\label{theorem:newMockThetaIdentitiespP512m0ell2r} For $(p,p^{\prime})=(5,12)$, $(m,\ell)=(0,2r)$, $r\in\{0,1,2,3,4,5\}$, we have
\begin{align*}
(q)_{\infty}^3\mathcal{C}_{0,2r}^{2/5}(q)
&=-q^{\frac{1}{2}r^2-\frac{5}{2}r+1}j(q^{2+4r};q^{12})j(-q^{1+2r};q^{8})\frac{J_{1}}{J_{4}}\\
&\qquad   - q^{6-4r}
    \times\Big (  j(-q^{17+10r};q^{120} )
 -  q^{11-2r}j(-q^{-7+10r};q^{120})\Big )
     \times q^2\phi_{10}(-q^2)\\
&\qquad     +q^{3-3r}
    \times\Big (  j(-q^{29+10r};q^{120} )
 -  q^{22-4r}j(-q^{-19+10r};q^{120})\Big )
     \times \chi_{10}(q^4)\\
&\qquad   -q^{1-2r}
  \times\Big (  j(-q^{41+10r};q^{120} )
 -  q^{33-6r}j(-q^{-31+10r};q^{120})\Big )
    \times \left ( -\psi_{10}(-q^2)\right)\\
&\qquad     +q^{-r}
   \times\Big (  j(-q^{53+10r};q^{120} )
 -  q^{44-8r}j(-q^{-43+10r};q^{120})\Big )
    \times X_{10}(q^4). 
\end{align*}
\end{theorem}

\begin{theorem}\label{theorem:newMockThetaIdentitiespP512m2ell2r} For $(p,p^{\prime})=(5,12)$, $(m,\ell)=(2,2r)$, $r\in\{0,1,2,3,4,5\}$, we have
   \begin{align*}
(q)_{\infty}^3\mathcal{C}_{2,2r}^{2/5}(q)
&=q^{\frac{1}{2}r^2-\frac{3}{2}r+3}j(q^{2+4r};q^{12})j(-q^{5+2r};q^8)\frac{J_{1}}{J_{4}}\\
&\qquad  -q^{10-4r}     \times\left (  j(-q^{17+10r};q^{120} )
 -  q^{11-2r}j(-q^{-7+10r};q^{120})\right )   \times
\left (1-X_{10}(q^4)\right )\\
&\qquad    +q^{6-3r} 
    \times\left (  j(-q^{29+10r};q^{120} )
 -  q^{22-4r}j(-q^{-19+10r};q^{120})\right )
     \times
\left (1+\psi_{10}(-q^2)\right )\\
&\qquad     -q^{3-2r}    \times\left (  j(-q^{41+10r};q^{120} )
 -  q^{33-6r}j(-q^{-31+10r};q^{120})\right )
   \times
\left (1-\chi_{10}(q^4)\right )\\
&\qquad   +q^{1-r} 
    \times\left (  j(-q^{53+10r};q^{120} )
 -  q^{44-8r}j(-q^{-43+10r};q^{120})\right )
     \times
\left (1-q^2\phi_{10}(-q^2)\right ). 
 \end{align*}
\end{theorem}

In Section \ref{section:techPrelim}, we recall basic facts about theta functions and Appell functions.  We also state a general formula for a character evaluation.  In Section \ref{section:Ram}, we recall the Appell function forms for the four tenth-order mock theta functions and also present alternate forms that are more useful for the paper.  In Section \ref{section:fryeGarvan}, we use the methods of Frye and Garvan \cite{FG} to prove two families of theta function identities.  In Section \ref{section:newResults}, we prove our new results Theorems  \ref{theorem:newMockThetaIdentitiespP512m0ell2r}  and \ref{theorem:newMockThetaIdentitiespP512m2ell2r}.


\section{Technical preliminaries}\label{section:techPrelim}

Following from the definitions are the general identities:
\begin{subequations}
{\allowdisplaybreaks \begin{gather}
j(q^n x;q)=(-1)^nq^{-\binom{n}{2}}x^{-n}j(x;q), \ \ n\in\mathbb{Z},\label{equation:j-elliptic}\\
j(x;q)=j(q/x;q)=-xj(x^{-1};q)\label{equation:j-flip},\\
j(x;q)={J_1}j(x,qx,\dots,q^{n-1}x;q^n)/{J_n^n} \ \ {\text{if $n\ge 1$,}}\label{equation:1.10}\\
j(x^n;q^n)={J_n}j(x,\zeta_nx,\dots,\zeta_n^{n-1}x;q^n)/{J_1^n} \ \ {\text{if $n\ge 1$,}}\label{equation:1.12}\\
j(z;q)=\sum_{k=0}^{m-1}(-1)^k q^{\binom{k}{2}}z^k
j\big ((-1)^{m+1}q^{\binom{m}{2}+mk}z^m;q^{m^2}\big ),\label{equation:jsplit}\\
j(qx^3;q^3)+xj(q^2x^3;q^3)=j(-x;q)j(qx^2;q^2)/J_2={J_1j(x^2;q)}/{j(x;q)},\label{equation:quintuple}
\end{gather}}%
\end{subequations}
\noindent  where identity (\ref{equation:quintuple}) is the quintuple product identity.

We will need to use some elementary properties of Appell functions.
\begin{proposition}  For generic $x,z\in \mathbb{C}^*$
{\allowdisplaybreaks 
\begin{subequations}
\begin{gather}
m(x,z;q)=m(x,qz;q),\label{equation:mxqz-fnq-z}\\
m(x,z;q)=x^{-1}m(x^{-1},z^{-1};q),\label{equation:mxqz-flip}\\
m(qx,z;q)=1-xm(x,z;q).\label{equation:mxqz-fnq-x}
\end{gather}
\end{subequations}}%
\end{proposition}

We also recall a formula for the character in terms of theta functions:
\begin{proposition} \label{proposition:WeylKac}\cite[Proposition $2.4$]{BoMo2025P2} We have
\begin{align*}
\chi_{\ell}^{(p,p^{\prime})}(z;q)
&=z^{-\frac{\ell+1}{2}}q^{p\frac{(\ell+1)^2}{4p^{\prime}}}
\frac{j(-q^{p(\ell+1)+pp^{\prime}}z^{-p^{\prime}};q^{2pp^{\prime}})
-z^{\ell+1}j(-q^{-p(\ell+1)+pp^{\prime}}z^{-p^{\prime}};q^{2pp^{\prime}}) }
{z^{-\frac{1}{2}}q^{\frac{1}{8}}j(z;q)}.
\end{align*}
\end{proposition}


\section{Ramanujan's classical mock theta functions}\label{section:Ram}
From \cite[Section 5]{HM}, we have the following classical mock theta functions in Appell function form.

\noindent {\bf `10th order' functions}
{\allowdisplaybreaks \begin{align}
{\phi_{10}}(q)&=\sum_{n\ge 0}\frac{q^{\binom{n+1}{2}}}{(q;q^2)_{n+1}}
=-q^{-1}m(q,q;q^{10})-q^{-1}m(q,q^2;q^{10})
\label{equation:10th-phi(q)}\\
{\psi_{10}}(q)
&=\sum_{n\ge 0}\frac{q^{\binom{n+2}{2}}}{(q;q^2)_{n+1}}
=-m(q^3,q;q^{10})-m(q^3,q^{3};q^{10})
\label{equation:10th-psi(q)}\\
{X_{10}}(q)
&=\sum_{n\ge 0}\frac{(-1)^nq^{n^2}}{(-q;q)_{2n}}
=m(-q^2,q;q^{5})+m(-q^2,q^{4};q^{5})
\label{equation:10th-BigX(q)}\\
{\chi_{10}}(q)
&=\sum_{n\ge 0}\frac{(-1)^nq^{(n+1)^2}}{(-q;q)_{2n+1}}
=m(-q,q^{2};q^{5})+m(-q,q^{3};q^{5})
\label{equation:10th-chi(q)}
\end{align}}%

For the purposes of the paper, we will use alternate Appell function forms for the tenth-order mock theta functions:
\begin{proposition}\label{proposition:alternat10thAppellForms} We have
{\allowdisplaybreaks \begin{gather}
2m(-q^2,q^{10};q^{20})
=q^2\phi_{10}(-q^2)
-q^2\frac{J_{20}^2J_{6,20}}
{\overline{J}_{2,10}J_{4,20}}
\cdot \frac{J_{2}}{\overline{J}_{6,20}},
\label{equation:newAppell10thPhi}\\
2m(-q^4,q^{10};q^{20})
=\chi_{10}(q^4) 
-q^{4}\frac{J_{20}^2J_{2,20}}{\overline{J}_{4,10}J_{8,20}}
\cdot \frac{J_{2}}
{\overline{J}_{8,20}},
\label{equation:newAppell10thChi}\\
2m(-q^6,q^{10};q^{20})
=-\psi(-q^2)
 -q^2\frac{J_{20}^2 J_{2,20}}
{\overline{J}_{4,10}J_{8,20}}
\cdot \frac{J_{2}}
{\overline{J}_{2,20}},
\label{equation:newAppell10thPsi}\\
2m(-q^8,q^{10};q^{20})
=X_{10}(q^4)
-\frac{J_{20}^2J_{6,20}}
{\overline{J}_{2,10}J_{4,20}}
\cdot \frac{J_{2}}
{\overline{J}_{4,20}}.
\label{equation:newAppell10thX}
\end{gather}}%
\end{proposition}
\begin{proof}  Use \cite[Proposition $5.2$]{BoMo2025P2} with the substitution $q\to q^2$.
\end{proof}


\section{Families of theta function identities through Frye and Garvan}
\label{section:fryeGarvan}

For the quantum-number zero, even-spin $2/3$-level mock theta conjecture-like identities found in Theorem \ref{theorem:newMockThetaIdentitiespP512m0ell2r}, we need to prove
\begin{proposition}
\label{proposition:masterThetaIdentitypP512m0ell2r} 
For $r\in\{0,1,2,3 \}$, we have
{\allowdisplaybreaks \begin{align*}
&-q^{\frac{1}{2}r^2-\frac{5}{2}r+1}j(q^{2+4r};q^{12})j(-q^{1+2r};q^{8})\frac{J_{1}}{J_{4}}\\
&\qquad =(-1)^{\kappa(r)}\frac{(q)_{\infty}^3}{J_{10,20}}\frac{j(-q^{10r+65};q^{120})}{J_{1,4}}\\
&\qquad \qquad  + q^{6-4r}
    \times\Big (  j(-q^{17+10r};q^{120} )
 -  q^{11-2r}j(-q^{-7+10r};q^{120})\Big )
\times q^2\frac{J_{20}^2J_{6,20}}
{\overline{J}_{2,10}J_{4,20}}
\cdot \frac{J_{2}}{\overline{J}_{6,20}}  \\
&\qquad  \qquad    -q^{3-3r}
    \times\Big (  j(-q^{29+10r};q^{120} )
 -  q^{22-4r}j(-q^{-19+10r};q^{120})\Big )
       \times 
q^{4}\frac{J_{20}^2J_{2,20}}{\overline{J}_{4,10}J_{8,20}}
\cdot \frac{J_{2}}
{\overline{J}_{8,20}} \\
&\qquad \qquad +q^{1-2r}
  \times\Big (  j(-q^{41+10r};q^{120} )
 -  q^{33-6r}j(-q^{-31+10r};q^{120})\Big )     \times 
 q^2\frac{J_{20}^2 J_{2,20}}
{\overline{J}_{4,10}J_{8,20}}
\cdot \frac{J_{2}}
{\overline{J}_{2,20}}\\
&\qquad \qquad    -q^{-r}
   \times\Big (  j(-q^{53+10r};q^{120} )
 -  q^{44-8r}j(-q^{-43+10r};q^{120})\Big )   \times
\frac{J_{20}^2J_{6,20}}
{\overline{J}_{2,10}J_{4,20}}
\cdot \frac{J_{2}}
{\overline{J}_{4,20}},
\end{align*}}%
where $\kappa(r):=\begin{cases} 0 & \textup{if} \ r=0,1,4,5,\\
1 & \textup{if} \ r=2,3. \end{cases}$
\end{proposition}

For the quantum-number two, even-spin  $2/5$-level mock theta conjecture-like identities found in Theorem \ref{theorem:newMockThetaIdentitiespP512m2ell2r}, we need to prove
\begin{proposition}
\label{proposition:masterThetaIdentitypP512m2ell2r}
For $r\in\{0,1,2,3,4,5 \}$, we have
{\allowdisplaybreaks  \begin{align*}
& q^{\frac{1}{2}r^2-\frac{3}{2}r+3}j(q^{2+4r};q^{12})j(-q^{5+2r};q^8)\frac{J_{1}}{J_{4}}\\
&\qquad =(-1)^{\kappa(r)}q^{15-5r}\frac{(q)_{\infty}^3}{J_{10,20}}
\frac{j(-q^{10r+5};q^{120}) }
{J_{1,4}}\\
&\qquad \qquad   -q^{10-4r}     \times\Big (  j(-q^{17+10r};q^{120} )
 -  q^{11-2r}j(-q^{-7+10r};q^{120})\Big )
     \times
\frac{J_{20}^2J_{6,20}}
{\overline{J}_{2,10}J_{4,20}}
\cdot \frac{J_{2}}
{\overline{J}_{4,20}} \\
&\qquad   \qquad  +q^{6-3r} 
    \times\Big (  j(-q^{29+10r};q^{120} )
 -  q^{22-4r}j(-q^{-19+10r};q^{120})\Big )
     \times
q^2\frac{J_{20}^2 J_{2,20}}
{\overline{J}_{4,10}J_{8,20}}
\cdot \frac{J_{2}}
{\overline{J}_{2,20}} \\
&\qquad  \qquad    -q^{3-2r}    \times\Big (  j(-q^{41+10r};q^{120} )
 -  q^{33-6r}j(-q^{-31+10r};q^{120})\Big )
     \times
q^{4}\frac{J_{20}^2J_{2,20}}{\overline{J}_{4,10}J_{8,20}}
\cdot \frac{J_{2}}
{\overline{J}_{8,20}} \\
&\qquad \qquad  +q^{1-r} 
    \times\Big (  j(-q^{53+10r};q^{120} )
 -  q^{44-8r}j(-q^{-43+10r};q^{120})\Big )
     \times
q^2\frac{J_{20}^2J_{6,20}}
{\overline{J}_{2,10}J_{4,20}}
\cdot \frac{J_{2}}{\overline{J}_{6,20}},  
 \end{align*}}%
where $\kappa(r):=\begin{cases} 0 & \textup{if} \ r=0,1,4,5,\\
1 & \textup{if} \ r=2,3. \end{cases}$
\end{proposition}

\begin{proof}[Proof of Propositions \ref{proposition:masterThetaIdentitypP512m0ell2r} and \ref{proposition:masterThetaIdentitypP512m2ell2r}] The proofs are a straightforward application of the methods of \cite{FG} and are analogous to the proofs of \cite[Propositions 6.1, 6.2]{BoMo2025P2} so they will be omitted.
\end{proof}

\section{The two $2/5$-level mock theta conjecture-like identities}\label{section:newResults}

We start with some general preliminaries.
\begin{proposition} \label{proposition:polarFinite25} We have
{\allowdisplaybreaks \begin{align*}
(q)_{\infty}^3&\chi_{2r}^{2/5} (z;q)\\
&=(q)_{\infty}^3C_{0,2r}^{2/5}(q)j(-z^{2}q^{10};q^{20})
+z^{-1}q^{\frac{5}{2}}(q)_{\infty}^3C_{2,2r}^{2/5}(q)j(-z^{2};q^{20})\\
&\qquad -q^{-\frac{1}{8}+\frac{5(2r+1)^2}{48}}q^{10-5r}j(-q^{10}z^{2};q^{20})\\
&\qquad \qquad   \times \sum_{m=1}^{4}(-1)^{m}q^{\binom{m+1}{2}+m(r-5)} \\
&\qquad \qquad \qquad    \times\Big (  j(-q^{12m+5(2r+1)};q^{120} )
 -  q^{12m-m(2r+1)}j(-q^{-12m+5(2r+1)};q^{120})\Big )\\
&\qquad \qquad \qquad    \times
\Big (m(-q^{2m},-q^{10}z^{-2};q^{20}) 
 + m(-q^{2m},-q^{10}z^{2};q^{20})\Big )\\
&\qquad -q^{-\frac{1}{8}+\frac{5(2r+1)^2}{48}}q^{10-5(r-1)}z^{-1}
j(-z^{2};q^{20})\\
&\qquad \qquad   \times \sum_{m=1}^{4}(-1)^{m}q^{\binom{m+1}{2}+m(r-5)} \\
&\qquad \qquad \qquad    \times\Big (  j(-q^{12m+5(2r+1)};q^{120} )
 -  q^{12m-m(2r+1)}j(-q^{-12m+5(2r+1)};q^{120})\Big )\\
&\qquad \qquad \qquad    \times
\Big (q^{m-10}m(-q^{2m-10},-q^{20}z^{-2};q^{20}) 
 + q^{-m}m(-q^{2m+10},-z^{2};q^{20})\Big ).
\end{align*}}%
\end{proposition}

\begin{lemma}\label{lemma:polarFinite25AppellVanish} We have
{\allowdisplaybreaks \begin{gather}
\lim_{z\to i} j(-z^2;q^{20}) \left ( q^{m-10}m(-q^{2m-10},-q^{20}z^{-2};q^{20}) 
 + q^{-m}m(-q^{2m+10},-z^{2};q^{20})\right ) =0,\label{equation:m0AppellVanish}\\
\lim_{z\to iq^5} j(-z^2q^{10};q^{20}) \left ( m(-q^{2m},-q^{10}z^{-2};q^{20}) 
 + m(-q^{2m},-q^{10}z^{2};q^{20})\right ) =0.\label{equation:m2AppellVanish}
\end{gather}}%
\end{lemma}

For $z=i$ and $z=iq^5$, we then need to evaluate the character on the left-hand side of Proposition \ref{proposition:polarFinite25}.  Then it is straightforward to solve for the string function.
\begin{proposition}\label{proposition:weylKac25ell2rzVal} We have
\begin{gather}
\chi_{2r}^{(5,12)}(i;q)
=(-1)^{\kappa(r)}q^{-\frac{1}{8}+\frac{5(2r+1)^2}{48}}j(-q^{10r+65};q^{120})\frac{J_{2}}{J_{1}J_{4}},
 \label{equation:weylKac25ell2rzVal1}\\
\chi_{2r}^{(5,12)}(iq^{5};q)
=-(-1)^{\kappa(r)}iq^{10-5r}q^{-\frac{1}{8}+\frac{5(2r+1)^2}{48}}j(-q^{10r+5};q^{120})
\frac{J_{2}}{J_{1}J_{4}}. 
 \label{equation:weylKac25ell2rzVal2}
\end{gather}
 where  $\kappa(r):=\begin{cases} 0 & \textup{if} \ r=0,1,4,5,\\
1 & \textup{if} \ r=2,3. \end{cases}$
\end{proposition}

\begin{proof}[Proof of Proposition \ref{proposition:polarFinite25} ]
We specialize Theorem \ref{theorem:generalPolarFinite} with  $p=5$, $j=2$ to get
{\allowdisplaybreaks \begin{align*}
&\chi_{2r}^{2/5} (z;q)\\
&=\sum_{s=0}^{1}z^{-s}q^{\frac{5}{2}s^2}C_{2s,2r}^{2/5}(q)j(-z^{2}q^{10-10s};q^{20})\\
&\qquad -\frac{1}{(q)_{\infty}^3}\sum_{s=0}^{1}
q^{-\frac{1}{8}+\frac{5(2r+1)^2}{48}}q^{10-5(r-s)}z^{-s}
j(-q^{10-10s}z^{2};q^{20})\\
&\qquad \qquad   \times \sum_{m=1}^{4}(-1)^{m}q^{\binom{m+1}{2}+m(r-5)} \\
&\qquad \qquad \qquad    \times\Big (  j(-q^{12m+5(2r+1)};q^{120} )
 -  q^{12m-m(2r+1)}j(-q^{-12m+5(2r+1)};q^{120})\Big )\\
&\qquad \qquad \qquad    \times
\Big (q^{ms-10s}m(-q^{2m-10s},-q^{10+10s}z^{-2};q^{20}) 
 + q^{-ms}m(-q^{2m+10s},-q^{10-10s}z^{2};q^{20})\Big ).
\end{align*}}%
 We then multiply by $(q)_{\infty}^3$ and expand the sum over $s$.
\end{proof}

\begin{proof}[Proof of Lemma \ref{lemma:polarFinite25AppellVanish}]
We prove (\ref{equation:m0AppellVanish}).  Let us first use (\ref{equation:mxqz-flip}) and (\ref{equation:mxqz-fnq-z}) to get
{\allowdisplaybreaks \begin{align*}
\lim_{z\to i}& j(-z^2;q^{20}) \left ( q^{m-10}m(-q^{2m-10},-q^{20}z^{-2};q^{20}) 
 + q^{-m}m(-q^{2m+10},-z^{2};q^{20})\right )  \\
 &=\lim_{z\to i} j(-z^2;q^{20}) \left ( -q^{-m}m(-q^{-2m+10},-q^{-20}z^{2};q^{20}) 
 + q^{-m}m(-q^{2m+10},-z^{2};q^{20})\right )\\
  &=-q^{-m}\left ( \sum_{k}\frac{(-1)^kq^{20\binom{k}{2}}}{1+q^{20(k-1)}q^{-2m+10}}
 - \sum_{k}\frac{(-1)^kq^{20\binom{k}{2}}}{1+q^{20(k-1)}q^{2m+10}}\right )\\
  &=q^{-m}\left ( \sum_{k}\frac{(-1)^kq^{20\binom{k+1}{2}}}{1+q^{20k}q^{-2m+10}}
 - \sum_{k}\frac{(-1)^kq^{20\binom{k+1}{2}}}{1+q^{20k}q^{2m+10}}\right )\\
  &=q^{-m}\left (\frac{J_{20}^3}{j(-q^{-2m+10};q^{20})}-\frac{J_{20}^3}{j(-q^{2m+10};q^{20})}\right )=0.
\end{align*}}%

We prove (\ref{equation:m2AppellVanish}).   Let us first use (\ref{equation:mxqz-flip}), (\ref{equation:j-flip}), and(\ref{equation:mxqz-fnq-z}) to get
{\allowdisplaybreaks \begin{align*}
\lim_{z\to iq^{5}} &j(-z^2q^{10};q^{20}) \left ( m(-q^{2m},-q^{10}z^{-2};q^{20}) 
 + m(-q^{2m},-q^{10}z^{2};q^{20})\right )\\
 &=\lim_{z\to iq^{5}} j(-z^{-2}q^{10};q^{20}) \left ( m(-q^{2m},-q^{10}z^{-2};q^{20}) 
 -q^{-2m} m(-q^{-2m},-q^{10}z^{-2};q^{20})\right )\\
  &=\left (\sum_{k}\frac{(-1)^kq^{20\binom{k}{2}}}{1+q^{20(k-1)}q^{2m}}
 - q^{-2m}\sum_{k}\frac{(-1)^kq^{20\binom{k}{2}}}{1+q^{20(k-1)}q^{-2m}}\right )\\
  &=-\left ( \sum_{k}\frac{(-1)^kq^{20\binom{k+1}{2}}}{1+q^{20k}q^{2m}}
 -q^{-2m} \sum_{k}\frac{(-1)^kq^{20\binom{k+1}{2}}}{1+q^{20k}q^{-2m}}\right )\\
  &=-\left (\frac{J_{20}^3}{j(-q^{2m};q^{20})}-q^{-2m}\frac{J_{20}^3}{j(-q^{-2m};q^{20})}\right )=0.\qedhere
\end{align*}}%

\end{proof}

\begin{proof}[Proof of Proposition \ref{proposition:weylKac25ell2rzVal}]  
We prove (\ref{equation:weylKac25ell2rzVal1}).  Taking the appropriate specialization of Proposition \ref{proposition:WeylKac} gives
\begin{equation}
\chi_{2r}^{(5,12)}(z;q)
=z^{-r}q^{-\frac{1}{8}+5\frac{(2r+1)^2}{48}}\frac{j(-q^{5(2r+1)+60}z^{-12};q^{120})
-z^{2r+1}j(-q^{-5(2r+1)+60}z^{-12};q^{120}) }
{j(z;q)}.\label{equation:weylKac25genz}
\end{equation}
Specializing (\ref{equation:weylKac25genz}) to $z=i$ and using (\ref{equation:j-flip}) gives
\begin{equation*}
\chi_{2r}^{(5,12)}(i;q)
=i^{-r}q^{-\frac{1}{8}+\frac{5(2r+1)^2}{48}}\frac{j(-q^{10r+65};q^{120})
(1-i(-1)^{r} )}
{j(i;q)}.
\end{equation*}
Using the Jacobi triple product identity (\ref{equation:JTPid}) allows us to write
\begin{equation*}
\frac{1}{j(i;q)}=\frac{1}{(1-i)(iq)_{\infty}(-iq)_{\infty}(q)_{\infty}}
=\frac{1+i}{2}\frac{1}{(-q^2;q^2)_{\infty}J_{1}}
=\frac{1+i}{2}\frac{J_{2}}{J_{1}J_{4}}.
\end{equation*}
It is then straightforward to show that
\begin{align*}
i^{-r}\frac{(1+i)(1-(-1)^ri)}{2}=(-1)^{\kappa(r)}, 
\qquad  \kappa(r):=\begin{cases} 0 & \textup{if} \ r=0,1,4,5\\
1 & \textup{if} \ r=2,3.
\end{cases}
\end{align*}

We prove (\ref{equation:weylKac25ell2rzVal2}).  Specializing (\ref{equation:weylKac25genz}) to $z=iq^5$ gives
\begin{equation*}
\chi_{2r}^{(5,12)}(iq^{5};q)
=i^{-r}q^{-5r}q^{-\frac{1}{8}+\frac{5(2r+1)^2}{48}}\frac{j(-q^{10r+5};q^{120})
-i(-1)^{r}q^{10r+5}j(-q^{-10r-5};q^{120}) }
{j(iq^{5};q)}
\end{equation*}
Using \ref{equation:j-flip} and \ref{equation:j-elliptic} and arguing as before yields
{\allowdisplaybreaks \begin{align*}
\chi_{2r}^{(5,12)}(iq^{5};q)
&=i^{-r}q^{-5r}q^{-\frac{1}{8}+5\frac{(2r+1)^2}{48}}\frac{j(-q^{10r+5};q^{120})
-i(-1)^{r}j(-q^{10r+5};q^{120}) }
{j(iq^{5};q)}\\
&=(-1)^{5}q^{\binom{5}{2}}i^{5}i^{-r}q^{-5r}q^{-\frac{1}{8}+5\frac{(2r+1)^2}{48}}\frac{j(-q^{10r+5};q^{120}) }
{J_{1,4}}\frac{1+i}{2}\left (1-i(-1)^{r} \right ) \\
&=-iq^{10-5r}q^{-\frac{1}{8}+5\frac{(2r+1)^2}{48}}\frac{j(-q^{10r+5};q^{120}) }
{J_{1,4}}\frac{1+i}{2}i^{-r}\left (1-i(-1)^{r} \right ),
\end{align*}}%
and the result follows.
\end{proof}


\subsection{Quantum number $m=0$, spin $\ell=2r$}   
We prove Theorem \ref{theorem:newMockThetaIdentitiespP512m0ell2r}.  We specialize Proposition \ref{proposition:polarFinite25} to $z=i$.   Limit (\ref{equation:m0AppellVanish}) of Lemma \ref{lemma:polarFinite25AppellVanish} indicates that the second sum over $m$ vanishes.  Dividing by $J_{10,20}$ then gives
\begin{align*}
\frac{(q)_{\infty}^3}{J_{10,20}}\chi_{2r}^{2/5} (i;q)
&=(q)_{\infty}^3C_{0,2r}^{2/5}(q)\\
&\qquad -q^{-\frac{1}{8}+\frac{5(2r+1)^2}{48}}q^{10-5r}  \times \sum_{m=1}^{4}(-1)^{m}q^{\binom{m+1}{2}+m(r-5)} \\
&\qquad  \qquad    \times\Big (  j(-q^{12m+5(2r+1)};q^{120} )
 -  q^{12m-m(2r+1)}j(-q^{-12m+5(2r+1)};q^{120})\Big )\\
&\qquad \qquad    \times 2m(-q^{2m},q^{10};q^{20}).
\end{align*}
Let us expand the sum over $m$.  This yields
{\allowdisplaybreaks \begin{align*}
\frac{(q)_{\infty}^3}{J_{10,20}}&\chi_{2r}^{2/5} (i;q)\\
&=(q)_{\infty}^3C_{0,2r}^{2/5}(q)\\
&\qquad -q^{-\frac{1}{8}+\frac{5(2r+1)^2}{48}}\\
&\qquad \qquad   \times\Big [  - q^{6-4r}
    \times\Big (  j(-q^{17+10r};q^{120} )
 -  q^{11-2r}j(-q^{-7+10r};q^{120})\Big )
     \times 2m(-q^{2},q^{10};q^{20})\\
&\qquad \qquad  \qquad  +q^{3-3r}
    \times\Big (  j(-q^{29+10r};q^{120} )
 -  q^{22-4r}j(-q^{-19+10r};q^{120})\Big )
    \times 2m(-q^{4},q^{10};q^{20})\\
&\qquad \qquad \qquad  -q^{1-2r}
  \times\Big (  j(-q^{41+10r};q^{120} )
 -  q^{33-6r}j(-q^{-31+10r};q^{120})\Big )
     \times 2m(-q^{6},q^{10};q^{20})\\
&\qquad \qquad  \qquad  +q^{-r}
   \times\Big (  j(-q^{53+10r};q^{120} )
 -  q^{44-8r}j(-q^{-43+10r};q^{120})\Big )   \times 2m(-q^{8},q^{10};q^{20})\Big ] .
\end{align*}}%

We replace the character using identity (\ref{equation:weylKac25ell2rzVal1}) of Proposition \ref{proposition:weylKac25ell2rzVal} and change the string function notation using (\ref{equation:mathCalCtoStringC}).  This produces

{\allowdisplaybreaks \begin{align*}
\frac{(q)_{\infty}^3}{J_{10,20}}&(-1)^{\kappa(r)}q^{-\frac{1}{8}+\frac{5(2r+1)^2}{48}}j(-q^{10r+65};q^{120})\frac{J_{2}}{J_{1}J_{4}}\\
&=(q)_{\infty}^3q^{-\frac{1}{8}+\frac{5(2r+1)^2}{48}}\mathcal{C}_{0,2r}^{2/5}(q)\\
&\qquad -q^{-\frac{1}{8}+\frac{5(2r+1)^2}{48}}\\
&\qquad \qquad \times \Big [  - q^{6-4r}
    \times\Big (  j(-q^{17+10r};q^{120} )
 -  q^{11-2r}j(-q^{-7+10r};q^{120})\Big )
     \times 2m(-q^{2},q^{10};q^{20})\\
&\qquad \qquad  \qquad   +q^{3-3r}
    \times\Big (  j(-q^{29+10r};q^{120} )
 -  q^{22-4r}j(-q^{-19+10r};q^{120})\Big )
     \times 2m(-q^{4},q^{10};q^{20})\\
&\qquad \qquad  \qquad -q^{1-2r}
  \times\Big (  j(-q^{41+10r};q^{120} )
 -  q^{33-6r}j(-q^{-31+10r};q^{120})\Big )
    \times 2m(-q^{6},q^{10};q^{20})\\
&\qquad \qquad \qquad   +q^{-r}
   \times\Big (  j(-q^{53+10r};q^{120} )
 -  q^{44-8r}j(-q^{-43+10r};q^{120})\Big )
    \times 2m(-q^{8},q^{10};q^{20})\Big ] .
\end{align*}}%
Cancelling out the fractional power of $q$ and rearranging terms gives
{\allowdisplaybreaks \begin{align*}
(q)_{\infty}^3&\mathcal{C}_{0,2r}^{2/5}(q)\\
&=(-1)^{\kappa(r)}\frac{(q)_{\infty}^3}{J_{10,20}}j(-q^{10r+65};q^{120})\frac{J_{2}}{J_{1}J_{4}}\\
&\qquad   - q^{6-4r}
    \times\Big (  j(-q^{17+10r};q^{120} )
 -  q^{11-2r}j(-q^{-7+10r};q^{120})\Big )
     \times 2m(-q^{2},q^{10};q^{20})\\
&\qquad     +q^{3-3r}
    \times\Big (  j(-q^{29+10r};q^{120} )
 -  q^{22-4r}j(-q^{-19+10r};q^{120})\Big )
     \times 2m(-q^{4},q^{10};q^{20})\\
&\qquad   -q^{1-2r}
  \times\Big (  j(-q^{41+10r};q^{120} )
 -  q^{33-6r}j(-q^{-31+10r};q^{120})\Big )
    \times 2m(-q^{6},q^{10};q^{20})\\
&\qquad    +q^{-r}
   \times\Big (  j(-q^{53+10r};q^{120} )
 -  q^{44-8r}j(-q^{-43+10r};q^{120})\Big )
    \times 2m(-q^{8},q^{10};q^{20}).
\end{align*}}%

Proposition \ref{proposition:alternat10thAppellForms} enables us to rewrite the Appell functions in terms of the four tenth-order mock theta functions

{\allowdisplaybreaks \begin{align*}
(q)_{\infty}^3\mathcal{C}_{0,2r}^{2/5}(q)
&=(-1)^{\kappa(r)}\frac{(q)_{\infty}^3}{J_{10,20}}j(-q^{10r+65};q^{120})\frac{J_{2}}{J_{1}J_{4}}\\
&\qquad   - q^{6-4r}
    \times\Big (  j(-q^{17+10r};q^{120} )
 -  q^{11-2r}j(-q^{-7+10r};q^{120})\Big )\\
&\qquad \qquad \qquad      \times \left (q^2\phi_{10}(-q^2)
-q^2\frac{J_{20}^2J_{6,20}}
{\overline{J}_{2,10}J_{4,20}}
\cdot \frac{J_{2}}{\overline{J}_{6,20}} \right ) \\
&\qquad     +q^{3-3r}
    \times\Big (  j(-q^{29+10r};q^{120} )
 -  q^{22-4r}j(-q^{-19+10r};q^{120})\Big )\\
&\qquad \qquad \qquad      \times \left ( \chi_{10}(q^4) 
-q^{4}\frac{J_{20}^2J_{2,20}}{\overline{J}_{4,10}J_{8,20}}
\cdot \frac{J_{2}}
{\overline{J}_{8,20}}\right ) \\
&\qquad   -q^{1-2r}
  \times\Big (  j(-q^{41+10r};q^{120} )
 -  q^{33-6r}j(-q^{-31+10r};q^{120})\Big )\\
&\qquad \qquad \qquad     \times \left ( -\psi(-q^2)
 -q^2\frac{J_{20}^2 J_{2,20}}
{\overline{J}_{4,10}J_{8,20}}
\cdot \frac{J_{2}}
{\overline{J}_{2,20}}\right )\\
&\qquad     +q^{-r}
   \times\Big (  j(-q^{53+10r};q^{120} )
 -  q^{44-8r}j(-q^{-43+10r};q^{120})\Big )\\
&\qquad \qquad  \qquad   \times \left ( X_{10}(q^4)
-\frac{J_{20}^2J_{6,20}}
{\overline{J}_{2,10}J_{4,20}}
\cdot \frac{J_{2}}
{\overline{J}_{4,20}} \right ).
\end{align*}}%

Upon expanding the bracket, we see that it remains to show that
{\allowdisplaybreaks \begin{align*}
&-q^{\frac{1}{2}r^2-\frac{5}{2}r+1}j(q^{2+4r};q^{12})j(-q^{1+2r};q^{8})\frac{J_{1}}{J_{4}}\\
&\qquad =(-1)^{\kappa(r)}\frac{(q)_{\infty}^3}{J_{10,20}}j(-q^{10r+65};q^{120})\frac{J_{2}}{J_{1}J_{4}}\\
&\qquad \qquad  + q^{6-4r}
    \times\Big (  j(-q^{17+10r};q^{120} )
 -  q^{11-2r}j(-q^{-7+10r};q^{120})\Big )
\times q^2\frac{J_{20}^2J_{6,20}}
{\overline{J}_{2,10}J_{4,20}}
\cdot \frac{J_{2}}{\overline{J}_{6,20}}  \\
&\qquad   \qquad   -q^{3-3r}
    \times\Big (  j(-q^{29+10r};q^{120} )
 -  q^{22-4r}j(-q^{-19+10r};q^{120})\Big )
       \times 
q^{4}\frac{J_{20}^2J_{2,20}}{\overline{J}_{4,10}J_{8,20}}
\cdot \frac{J_{2}}
{\overline{J}_{8,20}} \\
&\qquad  \qquad +q^{1-2r}
  \times\Big (  j(-q^{41+10r};q^{120} )
 -  q^{33-6r}j(-q^{-31+10r};q^{120})\Big )     \times 
 q^2\frac{J_{20}^2 J_{2,20}}
{\overline{J}_{4,10}J_{8,20}}
\cdot \frac{J_{2}}
{\overline{J}_{2,20}}\\
&\qquad \qquad   -q^{-r}
   \times\Big (  j(-q^{53+10r};q^{120} )
 -  q^{44-8r}j(-q^{-43+10r};q^{120})\Big )   \times
\frac{J_{20}^2J_{6,20}}
{\overline{J}_{2,10}J_{4,20}}
\cdot \frac{J_{2}}
{\overline{J}_{4,20}},
\end{align*}}%
but this is just Proposition \ref{proposition:masterThetaIdentitypP512m0ell2r}.


\subsection{Quantum number $m=2$, spin $\ell=2r$} We prove Theorem \ref{theorem:newMockThetaIdentitiespP512m2ell2r}.
We specialize Proposition \ref{proposition:polarFinite25} to $z=iq^{5}$.   Limit (\ref{equation:m2AppellVanish}) of Lemma \ref{lemma:polarFinite25AppellVanish} indicates that the first sum over $m$ vanishes.  Dividing by $J_{10,20}$ then gives
{\allowdisplaybreaks \begin{align*}
\frac{(q)_{\infty}^3}{J_{10,20}}&\chi_{2r}^{2/5} (iq^{5};q)\\
&=i^{-1}q^{-\frac{5}{2}}(q)_{\infty}^3C_{2,2r}^{2/5}(q)\\
&\qquad  -i^{-1}q^{-\frac{1}{8}+\frac{5(2r+1)^2}{48}}q^{10-5r}\\
&\qquad \qquad   \times \sum_{m=1}^{4}(-1)^{m}q^{\binom{m+1}{2}+m(r-5)} \\
&\qquad \qquad \qquad    \times\Big (  j(-q^{12m+5(2r+1)};q^{120} )
 -  q^{12m-m(2r+1)}j(-q^{-12m+5(2r+1)};q^{120})\Big )\\
&\qquad \qquad \qquad    \times
\Big (q^{m-10}m(-q^{2m-10},q^{10};q^{20}) 
 + q^{-m}m(-q^{2m+10},q^{10};q^{20})\Big ).
\end{align*}}%

We expand the sum over $m$ to get
{\allowdisplaybreaks \begin{align*}
\frac{(q)_{\infty}^3}{J_{10,20}}&\chi_{2r}^{2/5} (iq^{5};q)\\
&=i^{-1}q^{-\frac{5}{2}}(q)_{\infty}^3C_{2,2r}^{2/5}(q)\\
&\qquad  -i^{-1}q^{-\frac{1}{8}+\frac{5(2r+1)^2}{48}}q^{10-5r}\\
&\qquad \qquad  \times \Big [ -q^{-4+r}     \times\Big (  j(-q^{17+10r};q^{120} )
 -  q^{11-2r}j(-q^{-7+10r};q^{120})\Big )\\
&\qquad \qquad \qquad    \times
\Big (q^{-9}m(-q^{-8},q^{10};q^{20}) 
 + q^{-1}m(-q^{12},q^{10};q^{20})\Big )\\
&\qquad \qquad   +q^{-7+2r} 
    \times\Big (  j(-q^{29+10r};q^{120} )
 -  q^{22-4r}j(-q^{-19+10r};q^{120})\Big )\\
&\qquad \qquad \qquad    \times
\Big (q^{-8}m(-q^{-6},q^{10};q^{20}) 
 + q^{-2}m(-q^{14},q^{10};q^{20})\Big )\\
&\qquad \qquad    -q^{-9+3r}    \times\Big (  j(-q^{41+10r};q^{120} )
 -  q^{33-6r}j(-q^{-31+10r};q^{120})\Big )\\
&\qquad \qquad \qquad    \times
\Big (q^{-7}m(-q^{-4},q^{10};q^{20}) 
 + q^{-3}m(-q^{16},q^{10};q^{20})\Big )\\
&\qquad \qquad  +q^{-10+4r} 
    \times\Big (  j(-q^{53+10r};q^{120} )
 -  q^{44-8r}j(-q^{-43+10r};q^{120})\Big )\\
&\qquad \qquad \qquad    \times
\Big (q^{-6}m(-q^{-2},q^{10};q^{20}) 
 + q^{-4}m(-q^{18},q^{10};q^{20})\Big )\Big ]. 
\end{align*}}%

We replace the character using identity (\ref{equation:weylKac25ell2rzVal2}) of Proposition \ref{proposition:weylKac25ell2rzVal} and change the string function notation using (\ref{equation:mathCalCtoStringC}).  This produces

{\allowdisplaybreaks \begin{align*}
-(-1)^{\kappa(r)}&iq^{10-5r}q^{-\frac{1}{8}+\frac{5(2r+1)^2}{48}}j(-q^{10r+5};q^{120})
\frac{J_{2}}{J_{1}J_{4}}\\
&=i^{-1}q^{-\frac{5}{2}}(q)_{\infty}^3q^{-\frac{1}{8}+\frac{5(2r+1)^2}{48}}q^{-\frac{5}{2}}\mathcal{C}_{2,2r}^{2/5}(q)\\
&\qquad  -i^{-1}q^{-\frac{1}{8}+\frac{5(2r+1)^2}{48}}q^{10-5r}\\
&\qquad  \qquad \times \Big [ -q^{-4+r}     \times\Big (  j(-q^{17+10r};q^{120} )
 -  q^{11-2r}j(-q^{-7+10r};q^{120})\Big )\\
&\qquad \qquad \qquad  \qquad \times
\Big (q^{-9}m(-q^{-8},q^{10};q^{20}) 
 + q^{-1}m(-q^{12},q^{10};q^{20})\Big )\\
&\qquad \qquad  \qquad  +q^{-7+2r} 
    \times\Big (  j(-q^{29+10r};q^{120} )
 -  q^{22-4r}j(-q^{-19+10r};q^{120})\Big )\\
&\qquad \qquad \qquad \qquad   \times
\Big (q^{-8}m(-q^{-6},q^{10};q^{20}) 
 + q^{-2}m(-q^{14},q^{10};q^{20})\Big )\\
&\qquad \qquad  \qquad   -q^{-9+3r}    \times\Big (  j(-q^{41+10r};q^{120} )
 -  q^{33-6r}j(-q^{-31+10r};q^{120})\Big )\\
&\qquad \qquad \qquad  \qquad  \times
\Big (q^{-7}m(-q^{-4},q^{10};q^{20}) 
 + q^{-3}m(-q^{16},q^{10};q^{20})\Big )\\
&\qquad \qquad \qquad +q^{-10+4r} 
    \times\Big (  j(-q^{53+10r};q^{120} )
 -  q^{44-8r}j(-q^{-43+10r};q^{120})\Big )\\
&\qquad \qquad \qquad \qquad    \times
\Big (q^{-6}m(-q^{-2},q^{10};q^{20}) 
 + q^{-4}m(-q^{18},q^{10};q^{20})\Big )\Big ].
 \end{align*}}%
 
Multiplying both sides by $iq^5$ and cancelling out the fractional power of $q$ gives
{\allowdisplaybreaks \begin{align*}
(-1)^{\kappa(r)}&q^{15-5r}\frac{(q)_{\infty}^3}{J_{10,20}}
j(-q^{10r+5};q^{120})\frac{J_{2}}{J_{1}J_{4}}\\
&=(q)_{\infty}^3\mathcal{C}_{2,2r}^{2/5}(q)\\
&\qquad  -q^{15-5r} \Big [ -q^{-4+r}     \times\Big (  j(-q^{17+10r};q^{120} )
 -  q^{11-2r}j(-q^{-7+10r};q^{120})\Big )\\
&\qquad \qquad \qquad  \qquad \times
\Big (q^{-9}m(-q^{-8},q^{10};q^{20}) 
 + q^{-1}m(-q^{12},q^{10};q^{20})\Big )\\
&\qquad \qquad  \qquad  +q^{-7+2r} 
    \times\Big (  j(-q^{29+10r};q^{120} )
 -  q^{22-4r}j(-q^{-19+10r};q^{120})\Big )\\
&\qquad \qquad \qquad \qquad   \times
\Big (q^{-8}m(-q^{-6},q^{10};q^{20}) 
 + q^{-2}m(-q^{14},q^{10};q^{20})\Big )\\
&\qquad \qquad  \qquad   -q^{-9+3r}    \times\Big (  j(-q^{41+10r};q^{120} )
 -  q^{33-6r}j(-q^{-31+10r};q^{120})\Big )\\
&\qquad \qquad \qquad  \qquad  \times
\Big (q^{-7}m(-q^{-4},q^{10};q^{20}) 
 + q^{-3}m(-q^{16},q^{10};q^{20})\Big )\\
&\qquad \qquad \qquad +q^{-10+4r} 
    \times\Big (  j(-q^{53+10r};q^{120} )
 -  q^{44-8r}j(-q^{-43+10r};q^{120})\Big )\\
&\qquad \qquad \qquad \qquad    \times
\Big (q^{-6}m(-q^{-2},q^{10};q^{20}) 
 + q^{-4}m(-q^{18},q^{10};q^{20})\Big )\Big ].
 \end{align*}}%

We rewrite the first Appell function of each line using (\ref{equation:mxqz-flip}).  This gives
{\allowdisplaybreaks  \begin{align*}
(-1)^{\kappa(r)}&q^{15-5r}\frac{(q)_{\infty}^3}{J_{10,20}}
j(-q^{10r+5};q^{120})\frac{J_{2}}{J_{1}J_{4}}\\
&=(q)_{\infty}^3\mathcal{C}_{2,2r}^{2/5}(q)\\
&\qquad  -q^{15-5r}\Big [ -q^{-4+r}     \times\Big (  j(-q^{17+10r};q^{120} )
 -  q^{11-2r}j(-q^{-7+10r};q^{120})\Big )\\
&\qquad \qquad \qquad    \times
\Big (-q^{-1}m(-q^{8},q^{-10};q^{20}) 
 + q^{-1}m(-q^{12},q^{10};q^{20})\Big )\\
&\qquad \qquad   +q^{-7+2r} 
    \times\Big (  j(-q^{29+10r};q^{120} )
 -  q^{22-4r}j(-q^{-19+10r};q^{120})\Big )\\
&\qquad \qquad \qquad    \times
\Big (-q^{-2}m(-q^{6},q^{-10};q^{20}) 
 + q^{-2}m(-q^{14},q^{10};q^{20})\Big )\\
&\qquad \qquad    -q^{-9+3r}    \times\Big (  j(-q^{41+10r};q^{120} )
 -  q^{33-6r}j(-q^{-31+10r};q^{120})\Big )\\
&\qquad \qquad \qquad    \times
\Big (-q^{-3}m(-q^{4},q^{-10};q^{20}) 
 + q^{-3}m(-q^{16},q^{10};q^{20})\Big )\\
&\qquad \qquad  +q^{-10+4r} 
    \times\Big (  j(-q^{53+10r};q^{120} )
 -  q^{44-8r}j(-q^{-43+10r};q^{120})\Big )\\
&\qquad \qquad \qquad    \times
\Big (-q^{-4}m(-q^{2},q^{-10};q^{20}) 
 + q^{-4}m(-q^{18},q^{10};q^{20})\Big )\Big ]. 
 \end{align*}}%
 Rearranging terms yields
 {\allowdisplaybreaks  \begin{align*}
  (q)_{\infty}^3\mathcal{C}_{2,2r}^{2/5}(q)
&=(-1)^{\kappa(r)}q^{15-5r}\frac{(q)_{\infty}^3}{J_{10,20}}
j(-q^{10r+5};q^{120})\frac{J_{2}}{J_{1}J_{4}}\\
&\qquad  +q^{15-5r}\Big [ -q^{-4+r}     \times\Big (  j(-q^{17+10r};q^{120} )
 -  q^{11-2r}j(-q^{-7+10r};q^{120})\Big )\\
&\qquad \qquad \qquad    \times
\Big (-q^{-1}m(-q^{8},q^{-10};q^{20}) 
 + q^{-1}m(-q^{12},q^{10};q^{20})\Big )\\
&\qquad \qquad   +q^{-7+2r} 
    \times\Big (  j(-q^{29+10r};q^{120} )
 -  q^{22-4r}j(-q^{-19+10r};q^{120})\Big )\\
&\qquad \qquad \qquad    \times
\Big (-q^{-2}m(-q^{6},q^{-10};q^{20}) 
 + q^{-2}m(-q^{14},q^{10};q^{20})\Big )\\
&\qquad \qquad    -q^{-9+3r}    \times\Big (  j(-q^{41+10r};q^{120} )
 -  q^{33-6r}j(-q^{-31+10r};q^{120})\Big )\\
&\qquad \qquad \qquad    \times
\Big (-q^{-3}m(-q^{4},q^{-10};q^{20}) 
 + q^{-3}m(-q^{16},q^{10};q^{20})\Big )\\
&\qquad \qquad  +q^{-10+4r} 
    \times\Big (  j(-q^{53+10r};q^{120} )
 -  q^{44-8r}j(-q^{-43+10r};q^{120})\Big )\\
&\qquad \qquad \qquad    \times
\Big (-q^{-4}m(-q^{2},q^{-10};q^{20}) 
 + q^{-4}m(-q^{18},q^{10};q^{20})\Big )\Big ]. 
 \end{align*}}%
 
 Consolidating the lead $q$-terms gives
{\allowdisplaybreaks  \begin{align*}
  (q)_{\infty}^3\mathcal{C}_{2,2r}^{2/5}(q)
&=(-1)^{\kappa(r)}q^{15-5r}\frac{(q)_{\infty}^3}{J_{10,20}}
j(-q^{10r+5};q^{120})\frac{J_{2}}{J_{1}J_{4}}\\
&\qquad   -q^{10-4r}     \times\Big (  j(-q^{17+10r};q^{120} )
 -  q^{11-2r}j(-q^{-7+10r};q^{120})\Big )\\
&\qquad \qquad \qquad    \times
\Big (-m(-q^{8},q^{10};q^{20}) 
 + m(-q^{12},q^{10};q^{20})\Big )\\
&\qquad \qquad   +q^{6-3r} 
    \times\Big (  j(-q^{29+10r};q^{120} )
 -  q^{22-4r}j(-q^{-19+10r};q^{120})\Big )\\
&\qquad \qquad \qquad    \times
\Big (-m(-q^{6},q^{10};q^{20}) 
 + m(-q^{14},q^{10};q^{20})\Big )\\
&\qquad \qquad    -q^{3-2r}    \times\Big (  j(-q^{41+10r};q^{120} )
 -  q^{33-6r}j(-q^{-31+10r};q^{120})\Big )\\
&\qquad \qquad \qquad    \times
\Big (-m(-q^{4},q^{10};q^{20}) 
 + m(-q^{16},q^{10};q^{20})\Big )\\
&\qquad \qquad  +q^{1-r} 
    \times\Big (  j(-q^{53+10r};q^{120} )
 -  q^{44-8r}j(-q^{-43+10r};q^{120})\Big )\\
&\qquad \qquad \qquad    \times
\Big (-m(-q^{2},q^{10};q^{20}) 
 + m(-q^{18},q^{10};q^{20})\Big ). 
 \end{align*}}%

We rewrite the second Appell function in each line using (\ref{equation:mxqz-fnq-x}) and (\ref{equation:mxqz-flip}).  This brings us to
{\allowdisplaybreaks  \begin{align*}
  (q)_{\infty}^3\mathcal{C}_{2,2r}^{2/5}(q)
&=(-1)^{\kappa(r)}q^{15-5r}\frac{(q)_{\infty}^3}{J_{10,20}}
j(-q^{10r+5};q^{120})\frac{J_{2}}{J_{1}J_{4}}\\
&\qquad   -q^{10-4r}     \times\Big (  j(-q^{17+10r};q^{120} )
 -  q^{11-2r}j(-q^{-7+10r};q^{120})\Big )\\
&\qquad \qquad \qquad    \times
\Big (1-2m(-q^{8},q^{10};q^{20}) \Big )\\
&\qquad \qquad   +q^{6-3r} 
    \times\Big (  j(-q^{29+10r};q^{120} )
 -  q^{22-4r}j(-q^{-19+10r};q^{120})\Big )\\
&\qquad \qquad \qquad    \times
\Big (1-2m(-q^{6},q^{10};q^{20}) \Big )\\
&\qquad \qquad    -q^{3-2r}    \times\Big (  j(-q^{41+10r};q^{120} )
 -  q^{33-6r}j(-q^{-31+10r};q^{120})\Big )\\
&\qquad \qquad \qquad    \times
\Big (1-2m(-q^{4},q^{10};q^{20}) \Big )\\
&\qquad \qquad  +q^{1-r} 
    \times\Big (  j(-q^{53+10r};q^{120} )
 -  q^{44-8r}j(-q^{-43+10r};q^{120})\Big )\\
&\qquad \qquad \qquad    \times
\Big (1-2m(-q^{2},q^{10};q^{20}) \Big ). 
 \end{align*}}%

 Using the Proposition \ref{proposition:alternat10thAppellForms}, we rewrite the Appell functions in terms of the four tenth-order mock theta functions.  This gives
 
{\allowdisplaybreaks \begin{align*}
 (q)_{\infty}^3\mathcal{C}_{2,2r}^{2/5}(q)
&=(-1)^{\kappa(r)}q^{15-5r}\frac{(q)_{\infty}^3}{J_{10,20}}
j(-q^{10r+5};q^{120})\frac{J_{2}}{J_{1}J_{4}}\\
&\qquad   -q^{10-4r}     \times\Big (  j(-q^{17+10r};q^{120} )
 -  q^{11-2r}j(-q^{-7+10r};q^{120})\Big )\\
&\qquad \qquad \qquad    \times
\left( (1-\left ( X_{10}(q^4)
-\frac{J_{20}^2J_{6,20}}
{\overline{J}_{2,10}J_{4,20}}
\cdot \frac{J_{2}}
{\overline{J}_{4,20}} \right ) \right)\\
&\qquad \qquad   +q^{6-3r} 
    \times\Big (  j(-q^{29+10r};q^{120} )
 -  q^{22-4r}j(-q^{-19+10r};q^{120})\Big )\\
&\qquad \qquad \qquad    \times
\left (1-\left (-\psi(-q^2)
 -q^2\frac{J_{20}^2 J_{2,20}}
{\overline{J}_{4,10}J_{8,20}}
\cdot \frac{J_{2}}
{\overline{J}_{2,20}} \right ) \right )\\
&\qquad \qquad    -q^{3-2r}    \times\Big (  j(-q^{41+10r};q^{120} )
 -  q^{33-6r}j(-q^{-31+10r};q^{120})\Big )\\
&\qquad \qquad \qquad    \times
\left (1- \left (\chi_{10}(q^4) 
-q^{4}\frac{J_{20}^2J_{2,20}}{\overline{J}_{4,10}J_{8,20}}
\cdot \frac{J_{2}}
{\overline{J}_{8,20}} \right )  \right )\\
&\qquad \qquad  +q^{1-r} 
    \times\Big (  j(-q^{53+10r};q^{120} )
 -  q^{44-8r}j(-q^{-43+10r};q^{120})\Big )\\
&\qquad \qquad \qquad    \times
\left (1-\left (q^2\phi_{10}(-q^2)
-q^2\frac{J_{20}^2J_{6,20}}
{\overline{J}_{2,10}J_{4,20}}
\cdot \frac{J_{2}}{\overline{J}_{6,20}} \right )  \right ).
 \end{align*}}%
 
We see that it remains to show that
{\allowdisplaybreaks \begin{align*}
& q^{\frac{1}{2}r^2-\frac{3}{2}r+3}j(q^{2+4r};q^{12})j(-q^{5+2r};q^8)\frac{J_{1}}{J_{4}}\\
&\qquad =(-1)^{\kappa(r)}q^{15-5r}\frac{(q)_{\infty}^3}{J_{10,20}}
\frac{j(-q^{10r+5};q^{120}) }
{J_{1,4}}\\
&\qquad \qquad   -q^{10-4r}     \times\Big (  j(-q^{17+10r};q^{120} )
 -  q^{11-2r}j(-q^{-7+10r};q^{120})\Big )
     \times
\frac{J_{20}^2J_{6,20}}
{\overline{J}_{2,10}J_{4,20}}
\cdot \frac{J_{2}}
{\overline{J}_{4,20}} \\
&\qquad   \qquad  +q^{6-3r} 
    \times\Big (  j(-q^{29+10r};q^{120} )
 -  q^{22-4r}j(-q^{-19+10r};q^{120})\Big )
     \times
q^2\frac{J_{20}^2 J_{2,20}}
{\overline{J}_{4,10}J_{8,20}}
\cdot \frac{J_{2}}
{\overline{J}_{2,20}} \\
&\qquad  \qquad    -q^{3-2r}    \times\Big (  j(-q^{41+10r};q^{120} )
 -  q^{33-6r}j(-q^{-31+10r};q^{120})\Big )
     \times
q^{4}\frac{J_{20}^2J_{2,20}}{\overline{J}_{4,10}J_{8,20}}
\cdot \frac{J_{2}}
{\overline{J}_{8,20}} \\
&\qquad \qquad  +q^{1-r} 
    \times\Big (  j(-q^{53+10r};q^{120} )
 -  q^{44-8r}j(-q^{-43+10r};q^{120})\Big )
     \times
q^2\frac{J_{20}^2J_{6,20}}
{\overline{J}_{2,10}J_{4,20}}
\cdot \frac{J_{2}}{\overline{J}_{6,20}}, 
 \end{align*}}%
 but this is just Proposition \ref{proposition:masterThetaIdentitypP512m2ell2r}.


\section{Acknowledgments}

The work is supported by the Ministry of Science and Higher Education of the Russian
Federation (agreement no. 075-15-2025-343).

\end{document}